\documentclass[11pt]{article}
\usepackage{amsmath,amssymb,amscd,graphicx, color}
\usepackage{hyperref}
\usepackage{tensor}

\newcommand{\bbR}{{ \mathbb{R} }}

\newtheorem{lemma}{Lemma}

\newtheorem{theorem}{Theorem}

\newtheorem{remark}{Remark}
\numberwithin{equation}{section}

\newenvironment{proof}{{\bf Proof.}}{\hfill\fbox{}\par\vspace{.2cm}}

\begin{document}
\title{On Liouville type theorems for the stationary MHD and the Hall-MHD systems in $\bbR^3$}
\renewcommand{\baselinestretch}{1.2}
\renewcommand{\theequation}{\thesection.\arabic{equation}}
\author{Dongho Chae\footnote{E-mail: dchae@cau.ac.kr } , Junha Kim\footnote{E-mail: jha02@cau.ac.kr }  and J\"{o}rg Wolf \footnote{E-mail: jwolf2603@gmail.com }  \,\,\\
Department of Mathematics,\\Chung-Ang University\\
 Seoul 06974, Republic of Korea \\
  }
\date{}
\maketitle
\begin{abstract}
In this paper we prove a Liouville type theorem for the stationary MHD and the stationary Hall-MHD systems. Assuming suitable growth condition at infinity for the 
mean oscillations for the potential functions, we show that the solutions are trivial.  These results generalize the previous results obtained by two of the current authors  in [6].  To prove our main theorems  we use a refined iteration argument. 
  \noindent{\bf
keywords:} stationary magnetohydrodynamics equations, Hall-MHD system, Liouville type theorem 
\end{abstract}

\section{Introduction}
We consider the 3D stationary magnetohydrodynamics equations:
\begin{equation}\label{mhd_eq}
	\left\{
	\begin{array}{ll}
	-\Delta u + (u \cdot \nabla)u = -\nabla P + (B \cdot \nabla )B,\\
	-\Delta B + ( u\cdot \nabla)B = (B \cdot \nabla)u,\\
	\mathrm{div}\, u = \mathrm{div}\, B =0,
	\end{array}
	\right.
\end{equation}
which physically describe the steady-state of electrically conducting fluids, for example, plasmas. Here, $u=u(x)=(u_1,u_2,u_3)$ is the velocity field of the fluid, $B=B(x)=(B_1,B_2,B_3)$ is the magnetic field, and $P=P(x)$ is the pressure of the fluid. When $B\equiv0$, the system reduces to the stationary Navier--Stokes system. 

As the Liouville type problem for the stationary Navier-Stokes equations has been studied extensively recently(see \cite{C14,CY13,CW16,CJL21,GW78,KNSS09,KPR15,KTW17,S16,S18,SW20}), there also are many works on the Liouville type  problem  on the MHD equations (see \cite{LN20,YX20,ZYQ20} and references therein). Here, we focus on the results under the assumptions  in terms of the potential functions. We let $\Phi \in L^1_{loc}(\bbR^3;\bbR^{3\times 3})$ and $\Psi \in L^1_{loc}(\bbR^3;\bbR^{3\times 3})$ be the potential functions for the vector fields $u \in L^1_{loc}(\bbR^3)$ and $B \in L^1_{loc}(\bbR^3)$ respectively, if $\mathrm{div}\, \Phi = u$ and $\mathrm{div}\, \Psi = B$. In \cite{S16}, Seregin proved Liouville type theorems for the stationary Navier-Stokes equations under  the assumptions for the potential $\Phi \in BMO(\bbR^3)$ and for the velocity $v \in L^6(\bbR^3)$ (later in \cite{S18}, the velocity condition is dropped). After that, Chae and Wolf improved the previous result (see \cite{CW19}). For the MHD equations, Schulz obtained Liouville theorem in \cite{Sc19} under the conditions similar  to those in \cite{S16}. Recently, Chae and Wolf in \cite{CW21} proved the  similar theorem but under relaxed conditions than \cite{Sc19}. Moreover, they also obtained similar Liouville type result for the 3D Hall-MHD system
\begin{equation}\label{hall_mhd_eq}
	\left\{
	\begin{array}{ll}
	-\Delta u + (u \cdot \nabla)u = -\nabla P + (B \cdot \nabla )B,\\
	-\Delta B + ( u\cdot \nabla)B = (B \cdot \nabla)u + \nabla \times ((\nabla \times B ) \times B),\\
	\mathrm{div}\, u = \mathrm{div}\, B =0.
	\end{array}
	\right.
\end{equation}
We note that the system \eqref{hall_mhd_eq} is a physically  important  generalization of \eqref{mhd_eq}(see e.g. \cite{CDL14} for preliminary mathematical results and physical motivation of the Hall-MHD system). \\
\ \\
Our aim in this paper is to generalize the  main results in \cite{CW19, CW21}(and hence \cite{Sc19}).

\begin{theorem}\label{thm1}
	Let $(u,B,P)$ be a smooth solution to \eqref{mhd_eq}. Assume that there exist $\Phi,\Psi \in C^\infty(\bbR^3 ; \bbR^{3\times 3})$ such that $\nabla \cdot \Phi = u$, $\nabla \cdot \Psi = B$, and 
	\begin{equation}\label{AS}
		\bigg( \frac 1 {| B(r)|} \int_{B(r)} \big|\Phi - \Phi_{B(r)} \big|^s \,\mathrm{d}x \bigg)^{\frac 1s} + \bigg( \frac 1 {| B(r) |} \int_{B(r)} \big|\Psi - \Psi_{B(r)} \big|^s \,\mathrm{d}x \bigg)^{\frac 1s} \leq C r ^ {\frac 13 - \frac 1s}, \qquad r > 1
	\end{equation}
	for some $3 < s \leq 6$. Then $u \equiv B \equiv 0$.
\end{theorem}
\begin{remark}
	In the case  $B\equiv 0$ Theorem~\ref{thm1} reduces to  \cite[Theorem~1.1]{CW19}.
\end{remark}

\begin{theorem}\label{thm2}
	Let $(u,B,P)$ be a smooth solution to \eqref{hall_mhd_eq}. Assume that there exist $\Phi,\Psi \in C^\infty(\bbR^3 ; \bbR^{3\times 3})$ such that $\nabla \cdot \Phi = u$, $\nabla \cdot \Psi = B$, and \eqref{AS} for some $3<s\leq 6$. In addition,
	\begin{equation}\label{ASS}
		\bigg( \frac 1 {| B(r)|} \int_{B(r)} \big| B - B_{B(r)} \big|^p dx \bigg)^{\frac 1p} = o\Big(r^{( \frac{2s}3 + 1 )( \frac 13 - \frac 1p )}\Big) \qquad \mbox{as} \qquad r \rightarrow \infty
	\end{equation}
	for some $p>3$. Then $u \equiv B \equiv 0$.
\end{theorem}

\begin{remark}
	When $s=6$ and $p = 6$, the condition \eqref{ASS} reduces to
	\begin{equation*}
		\frac 1 {r^{8}} \int_{B(r)} \big| B - B_{B(r)} \big|^6 dx \rightarrow 0 \qquad \mbox{as} \qquad r \rightarrow \infty,
	\end{equation*}
	and we recover \cite[Theorem~1.3]{CW21}.
\end{remark}

\section{Preliminaries}

In this section, we introduce some notations and inequalities used throughout this paper. The following lemma is for Caccioppoli type inequalities which are essential for achieving our results. For the proof, we refer to \cite[Lemma~2.1 and Lemma~2.2]{CKW}.
\begin{lemma}\label{LEM3}
	Let $R>1$ and $f \in W^{1,2}(B(R);\bbR^3)$. For $0 < \rho <R$, let $\psi \in C^\infty_c(B(R))$ such that $0 \leq \psi \leq 1$ and $|\nabla \psi | \leq C/(R - \rho)$. Suppose that there exists $F \in W^{2,2}(B(R);\bbR^{3\times3})$ with $\nabla \cdot F = f$ and
	\begin{equation*}
	\bigg( \frac 1 {| B(r) |} \int_{B(r)} \big|F - F_{B(r)} \big|^s \,\mathrm{d}x \bigg)^{\frac 1s} \leq Cr^{\frac 13 - \frac 1s}, \qquad r>1
	\end{equation*}
	for some $3 < s \leq 6$. Then it holds
	\begin{equation}\label{C1}
	\| \psi^2 f \|_{L^2(B(R))}^2 \leq CR^{\frac {11}6 - \frac 1s} \| \psi \nabla f \|_{L^2(B(R))} + CR^{\frac {11}3 - \frac 2s} (R- \rho)^{-2}
	\end{equation}
	and
	\begin{equation}\label{C2}
	\| \psi^3 f \|_{L^3(B(R))}^3 \leq CR \| \psi \nabla f \|_{L^2(B(R))}^{\frac {18}{s+6}} + C R^{4-\frac 3s} (R - \rho)^{-3} + C R \big( (R-\rho)^{-1} \| \psi^2 f \|_{L^2(B(R))} \big)^{\frac {18}{s+6}}.
	\end{equation}
\end{lemma}

To prove Theorem~\ref{thm2}, it is important to estimate the Hall term in \eqref{hall_mhd_eq}, $\nabla \times ((\nabla \times B ) \times B)$, carefully. Here, we collect some simple inequalities which will help it. For a measurable set $E \subset \bbR^3$ and a cut-off function $\varphi \in C^\infty_c(\bbR^3)$ with $0 \leq \varphi \leq 1$, we use the notations
\begin{equation*}
	f_E := \frac{1}{|E|} \int_E f \,\mathrm{d}x, \qquad f_{\varphi} := \bigg( \int_{\bbR^3} \varphi \,\mathrm{d}x \bigg)^{-1} \int_{\bbR^3} f \varphi \,\mathrm{d}x
\end{equation*}
for $f \in L^1_{loc} (\bbR^3;\bbR^3)$.
\begin{lemma}\label{LEM1}
	Let $R>1$ and $\varphi \in C^\infty_c(B(R))$ satisfy $0 \leq \varphi \leq 1$ and $\| \varphi \|_{L^1(B(R))} \geq c|B(R)|$ for some $c>0$. Then for a measurable set $E \in \bbR^3$ with $\mathrm{supp}\, (\varphi) \subset E \subset B(R)$, we have
	\begin{equation}\label{2-1}
		\| f - f_{\varphi} \|_{L^p(E)} \leq \frac {c+1}c \| f - f_E \|_{L^p(E)}, \qquad 1 \leq p \leq \infty.
	\end{equation}
\end{lemma}
\begin{proof}
	Noting that $f_{\varphi} - \tau = (f - \tau)_{\varphi}$ for all $\tau \in \bbR$, we have
	\begin{equation*}
	\left\| f - f_{\varphi} \right\|_{L^p(E)} \leq \left\| f - f_E \right\|_{L^p(E)} + \left\| (f - f_E)_{\varphi} \right\|_{L^p(E)}.
	\end{equation*}
	Since H\"{o}lder's inequality implies that
	\begin{equation*}
		| (f - f_E)_{\varphi} | \leq \frac 1{c|B(R)|} \| f - f_E \|_{L^p(E)} \| \varphi \|_{L^{p'}(E)} \leq \frac {|E|^{1-\frac 1p}}{c|B(R)|} \| f - f_E \|_{L^p(E)}, \qquad 1 \leq p \leq \infty,
	\end{equation*}
	it follows that
	\begin{equation*}
		\left\| f - f_{\varphi} \right\|_{L^p(E)} \leq \bigg( 1 + \frac{|E|}{c |B(R)|} \bigg) \| f - f_E \|_{L^p(E)}.
	\end{equation*}
	This completes the proof.
\end{proof}
\begin{lemma}\label{LEM2}
	Let $B \in C^\infty(\bbR^3;\bbR^3)$ and $\Psi \in C^\infty(\bbR^3 ; \bbR^{3\times 3})$ satisfy $\nabla \cdot \Psi = B$ and \eqref{AS}. Let $R > 1$ and $\varphi \in C^\infty_c(B(R))$ satisfy $0 \leq \varphi \leq 1$, $\| \varphi \|_{L^1(B(R))} \geq c|B(R)|$, and $| \nabla \varphi | \leq C/R$ for some $c > 0$, $C>0$. Then we have
	\begin{equation}\label{2-2}
	\left| B_{\varphi} \right| \leq \frac C{c} R^{-\frac 23 - \frac 1s}.
	\end{equation}
\end{lemma}
\begin{proof}
	We clearly have that
	\begin{equation*}
		| B_{\varphi} | \leq \frac 1 {c|B(R)|} \bigg| \int_{B(R)} B \varphi \,\mathrm{d}x \bigg| = \frac 1 {c|B(R)|} \bigg| \int_{B(R)} \mathrm{div}\, \big( \Psi - \Psi_{B(R)} \big) \varphi \,\mathrm{d}x \bigg|.
	\end{equation*}
	Since integration by parts yields
	\begin{equation*}
		\frac 1 {c|B(R)|} \bigg| \int_{B(R)} \mathrm{div}\, \big( \Psi - \Psi_{B(R)} \big) \varphi \,\mathrm{d}x \bigg| \leq \frac 1 {c|B(R)|} \int_{B(R)} \big| \Psi - \Psi_{B(R)} \big| \big|\nabla \varphi\big| \,\mathrm{d}x ,
	\end{equation*}
	applying the H\"{o}lder's inequality with $|\nabla \varphi| \leq C/R$ we see that
	\begin{align*}
		\frac 1 {c|B(R)|} \bigg| \int_{B(R)} \mathrm{div}\, \big( \Psi - \Psi_{B(R)} \big) \varphi \,\mathrm{d}x \bigg| &\leq \frac {C}{c R} \bigg( \frac 1 {| B(R) |} \int_{B(R)} \big|\Phi - \Phi_{B(R)} \big| \,\mathrm{d}x \bigg) \\
		&\leq \frac {C}{c R} \bigg( \frac 1 {| B(R) |} \int_{B(R)} \big|\Phi - \Phi_{B(R)} \big|^s \,\mathrm{d}x \bigg)^{\frac 1s}.
	\end{align*}
	Thus, by \eqref{AS} we complete the proof.
\end{proof}

At the end of this section, we define a family of cut-off functions for simplicity. For $0<\tau<\tau'$, we let $\zeta=\zeta_{\tau,\tau'} \in C^\infty_c(B(\tau'))$ be a radially non-increasing scalar function such that $\zeta = 1$ on $B(\tau)$, $|\nabla \zeta | < 2/ (\tau'-\tau)$, and $|\nabla^2 \zeta | < 4/ (\tau'-\tau)^2$.

\section{Proof of Theorem~\ref{thm1}}

Firstly, we show that
\begin{equation}\label{EI}
	\int_{\bbR^3} (|\nabla u|^2 + |\nabla B|^2) \,\mathrm{d}x < \infty.
\end{equation}
Let $\varphi$ be a cut-off function in $C^\infty_c(\bbR^3)$. We multiply $\eqref{mhd_eq}_1$ and $\eqref{mhd_eq}_2$ by $u\varphi$ and $B\varphi$ respectively and integrate over $\bbR^3$. Then integration by parts with the divergence-free conditions yields that
\begin{equation}\label{zeta}
	\begin{gathered}
		\int_{\bbR^3} (|\nabla u|^2 + |\nabla B|^2)\varphi \,\mathrm{d}x = \frac 12 \int_{\bbR^3} (|u| ^2 + |B|^2) \Delta \varphi \,\mathrm{d}x + \frac 12 \int_{\bbR^3} (|u|^2 + |B|^2)u \cdot \nabla \varphi \,\mathrm{d}x \\
		- \int_{\bbR^3} (u \cdot B) (B \cdot \nabla) \varphi \,\mathrm{d}x + \int_{\bbR^3} \big(P - P_{B(\tau)}\big) u \cdot \nabla \varphi \,\mathrm{d}x,
	\end{gathered}
\end{equation}
where $\tau$ is any positive number. We set $R > \rho > 1$ and $\overline{R} = (R + \rho)/2$ and then take $\varphi = \zeta_{\rho, \overline{R}}$ in \eqref{zeta}. From the properties of $\zeta_{\rho, \overline{R}}$, H\"{o}lder's inequality, and Young's inequality we infer that
\begin{gather*}
	\int_{B(\rho)} (|\nabla u|^2 + |\nabla B|^2) \,\mathrm{d}x \leq C(R-\rho)^{-2} \int_{B(\overline{R})\setminus B(\rho)} (|u| ^2 + |B|^2) \,\mathrm{d}x \\
	+ C(R-\rho)^{-1} \int_{B(\overline{R})\setminus B(\rho)} (|u|^3 + |B|^3) \,\mathrm{d}x + C(R-\rho)^{-1} \int_{B(\overline{R})} \big|P - P_{B(\overline{R})}\big|^\frac 32 \,\mathrm{d}x.
\end{gather*}
Let $\psi = \zeta_{\overline{R}, R}$. Since $\psi = 1$ on $B(\overline{R})$, it follows
\begin{equation*}
	\int_{B(\rho)} (|\nabla u|^2 + |\nabla B|^2) \,\mathrm{d}x \leq C(I_1+I_2+I_3)
\end{equation*}
where
\begin{equation}\label{I123}
	\begin{aligned}
		I_1 &:= (R-\rho)^{-2} \int_{B(R)} (|\psi^2 u|^2 + |\psi^2 B|^2) \,\mathrm{d}x, \\
		I_2 &:= (R-\rho)^{-1} \int_{B(R)} (|\psi^3 u|^3 + |\psi^3 B|^3) \,\mathrm{d}x, \\
		I_3 &:= (R-\rho)^{-1} \int_{B(\overline{R})} \big|P - P_{B(\overline{R})} \big|^\frac 32 \,\mathrm{d}x.
	\end{aligned}
\end{equation}
To estimate $I_3$ first, we consider the functional $F \in W^{-1,\frac 32}(B(\overline{R}))$ such that
\begin {equation*}
	\langle F, \varphi \rangle = \int_{B(\overline{R})} (\nabla u - u\otimes u + B \otimes B) : \nabla \varphi \,\mathrm{d}x, \qquad \varphi \in W_0^{1,3}(B(\overline{R})).
\end{equation*}
Due to that $(u,B,P)$ solves the equation \eqref{mhd_eq}, we can verify $F =  -\nabla\big(P - P_{B(\overline{R})}\big)$ and $\langle F, \varphi \rangle = 0$ for $\varphi \in W_0^{1,3}(B(\overline{R}))$ with $\mathrm{div}\, \varphi = 0$. Thus, we can apply \cite[Lemma~2.1.1]{S01} and see that
\begin {equation*}
	\int_{B(\overline{R})} \big|P - P_{B(\overline{R})} \big|^\frac 32 \,\mathrm{d}x \leq C \| F \|_{W^{-1,\frac 32}(B(\overline{R}))}^{\frac 32}.
\end{equation*}
With the following estimate
\begin{align*}
	\| F \|_{W^{-1,\frac 32}(B(\overline{R}))}^{\frac 32} &\leq \| \nabla u - u \otimes u + B \otimes B \|_{L^{\frac 32}(B(\overline{R}))}^{\frac 32} \\
	&\leq CR^{\frac 34} \bigg( \int_{B(\overline{R})} |\nabla u|^2 \,\mathrm{d}x \bigg)^{\frac 34} + C \int_{B(\overline{R})} (|u|^3 + |B|^3) \,\mathrm{d}x
\end{align*}
we can obtain
\begin{equation*}
	I_3 \leq C R (R-\rho)^{-1} \bigg( \int_{B(R)} |\nabla u|^2 \,\mathrm{d}x \bigg)^{\frac 34} + C I_2.
\end{equation*}
Hence,
\begin{equation*}
	I_3 \leq \epsilon \int_{B(R)} |\nabla u|^2 \,\mathrm{d}x + C(\epsilon) R^{4}(R-\rho)^{4} + C I_2, \qquad \epsilon > 0.
\end{equation*}
Before estimating $I_2$, we notice that $\psi = \zeta_{\overline{R}, R}$ satisfies $0 \leq \psi \leq 1$ and $|\nabla \psi| \leq 4/(R-\rho)$, thus, with \eqref{AS}, Lemma~\ref{LEM3} is applicable. We use \eqref{C2} and have
\begin{gather*}
	(R-\rho)^{-1} \int_{B(R)} |\psi^3 u|^3 \,\mathrm{d}x \leq CR (R-\rho)^{-1} \| \psi \nabla u \|_{L^2(B(R))}^{\frac {18}{s+6}} + C R^{4-\frac 3s} (R - \rho)^{-4} \\
	+ C R(R-\rho)^{-1} \big( (R-\rho)^{-1} \| \psi^2 u \|_{L^2(B(R))} \big)^{\frac {18}{s+6}}.
\end{gather*}
By $R>1$, $R(R-\rho)^{-1}>1$, and Young's inequality, it follows
\begin{equation*}
	(R-\rho)^{-1} \int_{B(R)} |\psi^3 u|^3 \,\mathrm{d}x \leq \epsilon \| \psi \nabla u \|_{L^2(B(R))}^2 + C(\epsilon) R^{\frac {s+6}{s-3}}(R-\rho)^{-\frac {s+6}{s-3}} + I_1, \qquad \epsilon>0.
\end{equation*}
Repeating the above process for $B$ instead of $u$, we can have
\begin{equation*}
	(R-\rho)^{-1} \int_{B(R)} |\psi^3 B|^3 \,\mathrm{d}x \leq \epsilon \| \psi \nabla B \|_{L^2(B(R))}^2 + C(\epsilon) R^{\frac {s+6}{s-3}}(R-\rho)^{-\frac {s+6}{s-3}} + I_1, \qquad \epsilon>0.
\end{equation*}
Therefore, we obtain that
\begin{equation*}
	I_2 \leq \epsilon \int_{B(R)} (|\nabla u|^2 + |\nabla B|^2) \,\mathrm{d}x +C(\epsilon) R^{\frac {s+6}{s-3}}(R-\rho)^{-\frac {s+6}{s-3}} + I_1, \qquad \epsilon>0.
\end{equation*}
We continue estimating $I_1$. Using \eqref{C1} with $R>1$, we clearly have
\begin{align*}
	(R-\rho)^{-2} \int_{B(R)} |\psi^2 u|^2 \,\mathrm{d}x &\leq CR^{\frac {11}6 - \frac 1s} (R-\rho)^{-2} \| \psi \nabla u \|_{L^2(B(R))} + CR^{\frac {11}3 - \frac 2s} (R- \rho)^{-4} \\
	&\leq CR^2 (R-\rho)^{-2} \| \psi \nabla u \|_{L^2(B(R))} + CR^4 (R- \rho)^{-4}.
\end{align*}
Then by Young's inequality,
\begin{equation*}
	(R-\rho)^{-2} \int_{B(R)} |\psi^2 u|^2 \,\mathrm{d}x \leq \epsilon \| \psi \nabla u \|_{L^2(B(R))}^2 + C(\epsilon) R^4 (R- \rho)^{-4}, \qquad \epsilon > 0.
\end{equation*}
Similarly, we also obtain
\begin{equation*}
	(R-\rho)^{-2} \int_{B(R)} |\psi^2 B|^2 \,\mathrm{d}x \leq \epsilon \| \psi \nabla B \|_{L^2(B(R))}^2 + C(\epsilon) R^4 (R- \rho)^{-4}, \qquad \epsilon > 0,
\end{equation*}
thus,
\begin{equation*}
	I_1 \leq \epsilon \int_{B(R)} (|\nabla u|^2 + |\nabla B|^2) \,\mathrm{d}x + C(\epsilon) R^4 (R- \rho)^{-4}, \qquad \epsilon > 0.
\end{equation*}
Collecting the estimates for $I_1$, $I_2$, and $I_3$, we have with $R(R-\rho)^{-1}>1$ that
\begin{equation*}
	\int_{B(\rho)} (|\nabla u|^2 + |\nabla B|^2) \,\mathrm{d}x \leq \epsilon \int_{B(R)} (|\nabla u|^2 + |\nabla B|^2) \,\mathrm{d}x + C(\epsilon) R^{\frac {s+6}{s-3}}(R-\rho)^{-\frac {s+6}{s-3}}, \qquad \epsilon>0.
\end{equation*}
We fix $\epsilon <1$. Then, thanks to the iteration Lemma in \cite[V.Lemma~3.1]{G83}, we can deduce
\begin{equation*}
	\int_{B(\rho)} (|\nabla u|^2 + |\nabla B|^2) \,\mathrm{d}x \leq C R^{\frac {s+6}{s-3}}(R-\rho)^{-\frac {s+6}{s-3}}.
\end{equation*}
Taking $R = 2\rho$ and passing $\rho \rightarrow \infty$, we have \eqref{EI}.

Secondly, we show that
\begin{equation}\label{L3o1}
	r^{-1} \int _{B(2r) \setminus B(r)} (| u |^3 + | B |^3) \,\mathrm{d}x \to 0 \qquad \mbox{as} \qquad r\rightarrow \infty.
\end{equation}
Let $R>\rho>R/4>1$ and $\psi = \zeta_{\rho,R} - \zeta_{\rho/4, R/4}$. Since $\psi$ satisfies the assumptions for Lemma~\ref{LEM3}, with the assumption \eqref{AS} we use Lemma~\ref{LEM3} and have
\begin{gather*}
	\int_{B(R)} |\psi^3 u |^3 \,\mathrm{d}x \leq CR \| \psi \nabla u \|_{L^2(B(R))}^{\frac {18}{s+6}} + C R^{4-\frac 3s} (R - \rho)^{-3} \\
	+ C R \big( (R-\rho)^{-2} R^{\frac {11}6 - \frac 1s} \| \psi \nabla u \|_{L^2(B(R))} + R^{\frac {11}3 - \frac 2s} (R- \rho)^{-4} \big)^{\frac 9{s+6}}.
\end{gather*}
Applying Young's inequality, we have
\begin{gather*}
	\int_{B(R)} |\psi^3 u |^3 \,\mathrm{d}x \leq CR \| \psi \nabla u \|_{L^2(B(R))}^{\frac {18}{s+6}} + C R^{4-\frac 3s} (R - \rho)^{-3} \\
	+ C R \big( \| \psi \nabla u \|_{L^2(B(R))}^2 + R^{\frac {11}3 - \frac 2s} (R- \rho)^{-4} \big)^{\frac 9{s+6}}.
\end{gather*}
And by taking $\rho = 2r$ and $R = 4r$ for $r>1$, we can deduce that
\begin{equation*}
r^{-1} \int_{B(2r) \setminus B(r)} | u |^3 \,\mathrm{d}x \leq C \| \nabla u \|_{L^2(B(4r) \setminus B(r/2))}^{\frac {18}{s+6}} + C r^{-\frac 3s}.
\end{equation*}
Note that we can show the above inequality for $B$. Thus, we obtain \eqref{L3o1} due to \eqref{EI}.

Thirdly, we show
\begin{equation}\label{L3O1}
	r^{-1} \int _{B(r)} (| u |^3 +| B |^3) \,\mathrm{d}x \leq C, \qquad r > 1.
\end{equation}
By direct computation we see that
\begin{align*}
	r^{-1} \int_{B(r)} (| u |^3 +| B |^3) \,\mathrm{d}x &= \sum_{j=1}^\infty 2^{-j} (2^{-j}r)^{-1} \int_{B(2^{-(j-1)}r) \setminus B(2^{-j}r)} (| u |^3 +| B |^3) \,\mathrm{d}x \\
	&\leq \sup_{1/2 \leq \rho \leq r/2}\rho^{-1} \int _{B(2\rho) \setminus B(\rho)} (| u |^3 + | B |^3) \,\mathrm{d}x + \int_{B(1)} (| u |^3 + | B |^3) \,\mathrm{d}x.
\end{align*}
Hence, \eqref{L3o1} implies \eqref{L3O1}.

Now, we conclude $u \equiv B \equiv 0$. Let $r >1$. Inserting $\varphi = \zeta_{r,2r}$ into \eqref{zeta}, we infer that
\begin{gather*}
	\int_{B(r)} (|\nabla u|^2 + |\nabla B|^2) \,\mathrm{d}x \leq C r^{-2} \int_{B(2r)\setminus B(r)} (|u| ^2 + |B| ^2)\,\mathrm{d}x \\
+ C r^{-1} \!\int_{B(2r)\setminus B(r)} (|u|^3 + |B|^3)\,\mathrm{d}x + C r^{-1} \bigg( \!\int_{B(2r)} \big| P - P_{B(2r)} \big|^{\frac 32} \,\mathrm{d}x \!\bigg)^{\frac 23} \bigg( \!\int_{B(2r) \setminus B(r)} |u|^3 \,\mathrm{d}x \!\bigg)^{\frac 13}.
\end{gather*}
H\"{o}lder's inequality implies
\begin{equation*}
	r^{-2} \int_{B(2r)\setminus B(r)} (|u| ^2 + |B| ^2) \,\mathrm{d}x \leq C r^{-\frac 13} \bigg( r^{-1} \int_{B(2r)\setminus B(r)} (|u|^3 + |B|^3) \,\mathrm{d}x \bigg)^{\frac 23}.
\end{equation*}
By \eqref{L3O1} we have
\begin{equation*}
	r^{-2} \int_{B(2r)\setminus B(r)} (|u| ^2 + |B| ^2) \,\mathrm{d}x \to 0 \qquad \mbox{as} \qquad r \to \infty.
\end{equation*}
Also recall that \eqref{L3o1}
\begin{equation*}
	r^{-1} \int_{B(2r)\setminus B(r)} (|u|^3 + |B|^3)\,\mathrm{d}x \to 0 \qquad \mbox{as} \qquad r \to \infty.
\end{equation*}
As in the way we estimated $I_3$, we can show that
\begin{equation*}
	\int_{B(2r)} \big| P - P_{B(2r)} \big|^{\frac 32} \,\mathrm{d}x \leq Cr^{\frac 34} \bigg( \int_{B(2r)} |\nabla u|^2 \,\mathrm{d}x \bigg)^{\frac 34} + C \int_{B(2r)} (|u|^3 + |B|^3) \,\mathrm{d}x.
\end{equation*}
From \eqref{EI} and \eqref{L3O1} we see that
\begin{equation*}
	r^{-1} \int_{B(2r)} \big| P - P_{B(2r)} \big|^{\frac 32} \,\mathrm{d}x \leq C, \qquad r > 1.
\end{equation*}
Hence, it follows
\begin{equation*}
	r^{-1} \bigg( \int_{B(2r)} \big| P - P_{B(2r)} \big|^{\frac 32} \,\mathrm{d}x \bigg)^{\frac 23} \bigg( \int_{B(2r) \setminus B(r)} |u|^3 \,\mathrm{d}x \bigg)^{\frac 13} \leq C \bigg( r^{-1} \int_{B(2r) \setminus B(r)} |u|^3 \,\mathrm{d}x \bigg)^{\frac 13}.
\end{equation*}
And according to \eqref{L3o1}, we have
\begin{equation*}
	r^{-1} \bigg( \int_{B(2r)} \big| P - P_{B(2r)} \big|^{\frac 32} \,\mathrm{d}x \bigg)^{\frac 23} \bigg( \int_{B(2r) \setminus B(r)} |u|^3 \,\mathrm{d}x \bigg)^{\frac 13} \to 0 \qquad \mbox{as} \qquad r \to \infty.
\end{equation*}
This shows that
\begin{equation*}
	\int_{B(r)} (|\nabla u|^2 + |\nabla B|^2) \,\mathrm{d}x \to 0 \qquad \mbox{as} \qquad r \rightarrow \infty,
\end{equation*}
which implies that $u$ and $B$ must be constants. Thanks to \eqref{L3o1}, we finally obtain $u \equiv B \equiv 0$.

\section{Proof of Theorem~\ref{thm2}}

In this section, we use the notation
\begin{equation*}
	G(r) := \int_{B(r)} (|\nabla u|^2 + |\nabla B|^2) \,\mathrm{d}x
\end{equation*}
and
\begin{equation*}
	\Theta(r) := \frac 1 {r^{( \frac{2s}3 + 1 )( \frac 13 - \frac 1p )}} \bigg( \frac 1 {| B(r)|}\int_{B(r)} \big| B - B_{B(r)} \big|^p \,\mathrm{d}x \bigg)^{\frac 1p}
\end{equation*}
for $r>1$. Notice that by means of the condition \eqref{ASS}, it holds
\begin{equation}\label{ASSS}
	\lim_{r \rightarrow \infty} \Theta(r) = 0.
\end{equation}

We first show \eqref{EI}. For a cut-off function $\varphi \in C^\infty_c(\bbR^3)$, as we obtained \eqref{zeta} we have from \eqref{hall_mhd_eq} that
\begin{equation}\label{zeta2}
	\begin{gathered}
		\int_{\bbR^3} (|\nabla u|^2 + |\nabla B|^2)\varphi \,\mathrm{d}x = \frac 12 \int_{\bbR^3} (|u| ^2 + |B|^2) \Delta \varphi \,\mathrm{d}x + \frac 12 \int_{\bbR^3} (|u|^2 + |B|^2)u \cdot \nabla \varphi \,\mathrm{d}x \\
		- \int_{\bbR^3} (u \cdot B) (B \cdot \nabla) \varphi \,\mathrm{d}x - \int_{\bbR^3} ((\nabla \times B ) \times B) \cdot B \times \nabla \varphi \,\mathrm{d}x + \int_{\bbR^3} \big(P - P_{B(\tau)}\big) u \cdot \nabla \varphi \,\mathrm{d}x
	\end{gathered}
\end{equation}
for any $\tau>0$. Let $R > \rho > 1$ and $\overline{R} = (R + \rho)/2$ and insert $\varphi = \zeta_{\rho,\overline{R}}$ into \eqref{zeta2}. Then we can infer that
\begin{gather*}
	G(\rho) \leq C(R-\rho)^{-2} \int_{B(\overline{R})\setminus B(\rho)} (|u| ^2 + |B|^2) \,\mathrm{d}x + C(R-\rho)^{-1} \int_{B(\overline{R})\setminus B(\rho)} (|u|^3 + |B|^3) \,\mathrm{d}x \\
	+ C(R-\rho)^{-1} \int_{B(\overline{R})} \big|P - P_{B(\overline{R})}\big|^\frac 32 \,\mathrm{d}x + C(R-\rho)^{-1} \int_{B(\overline{R})} | \nabla B|| B|^2  \,\mathrm{d}x .
\end{gather*}
We define
\begin{equation*}
	I_4 := (R-\rho)^{-1} \int_{B(\overline{R})} | \nabla B|| B|^2  \,\mathrm{d}x .
\end{equation*}
With the notations \eqref{I123} we have that
\begin{equation*}
	\int_{B(\rho)} (|\nabla u|^2 + |\nabla B|^2) \,\mathrm{d}x \leq C(I_1+I_2+I_3+I_4).
\end{equation*}
In the proof of Theorem~\ref{thm1}, we already showed
\begin{equation}\label{I123}
	I_1 + I_2 + I_3 \leq \epsilon G(R) + C(\epsilon) R^{\frac {s+6}{s-3}}(R-\rho)^{-\frac {s+6}{s-3}}, \qquad \epsilon>0.
\end{equation}
We estimate $I_4$ by considering the cases $3<p\leq 4$ and $p > 4$ separately.

\medskip

\noindent \textit{(i) $3 < p \leq 4$ case:} 

\medskip

\noindent Using H\"{o}lder's inequality, we have
\begin{equation*}
	I_4 \leq (R-\rho)^{-1} \| B \|_{L^p(B(R))}^{\frac p{6-p}} \| B \|_{L^6(B(R))}^{\frac {3(4-p)}{6-p}} G(R)^{\frac 12}.
\end{equation*}
Let $\varphi = \zeta_{R/2,R}$. Then, since it is satisfied regardless of $R>1$ that $0 \leq \varphi \leq 1$, $\| \varphi \|_{L^1} \geq |B(R)|/8$, $|\nabla \varphi| \leq 4/R$, we can see by \eqref{2-1} and \eqref{2-2} that
\begin{equation*}
	\begin{aligned}
		\int_{B(R)} | B |^p \,\mathrm{d}x &\leq C \int_{B(R)} | B - B_{\varphi} |^p \,\mathrm{d}x + C |B_{\varphi}|^p |B(R)| \\
		&\leq C \int_{B(R)} \big| B - B_{B(R)} \big|^p \,\mathrm{d}x + C R^{-\frac {2p}3 - \frac ps + 3} .
	\end{aligned}
\end{equation*}
Thus, it follows
\begin{equation}\label{Lp_est}
	\int_{B(R)} | B |^p \,\mathrm{d}x \leq C R^{(\frac {2s}3 + 1)(\frac p3-1)+3}\Theta(R)^p + C R^{-\frac {2p}3 - \frac ps + 3}.
\end{equation}
Continuously, using $R^{(\frac {2s}3 + 1)(\frac p3-1)+3} = R^{6-p} R^{(\frac{2s}3 + 4)(\frac p3-1)}$ and $R^{-\frac {2p}3 - \frac ps + 3} \leq R^{6-p}$, we obtain
\begin{equation*}
	\int_{B(R)} |B|^p \,\mathrm{d}x \leq C R^{6-p} \big( R^{(\frac{2s}3 + 4)(\frac p3-1)}\Theta(R)^p + 1 \big).
\end{equation*}
As above we can also have
\begin{equation*}
	\int_{B(R)} |B|^6 \,\mathrm{d}x \leq C \int_{B(R)} \big| B - B_{B(R)} \big|^6 \,\mathrm{d}x + C R^{-1 - \frac 6s }.
\end{equation*}
By Poincaré-Sobolev inequality with $R^{-1-\frac 6s} \leq 1$,
\begin{equation}\label{PS}
	\int_{B(R)} |B|^6 \,\mathrm{d}x \leq C \bigg( \int_{B(R)} | \nabla B |^2 \,\mathrm{d}x \bigg)^{3} + C.
\end{equation}
Combining the above estimates, we can infer
\begin{equation*}
I_4 \leq C R (R - \rho)^{-1} \big( R^{(\frac{2s}3 + 4)(\frac p3-1)}\Theta(R)^p + 1 \big)^{\frac 1 {6-p}} \big( G(R) + 1 \big)^{\frac{3(4-p)}{2(6-p)}+\frac 12} .
\end{equation*}
Noting that
\begin{equation*}
	\frac{3(4-p)}{2(6-p)} + \frac 12  < 1, \qquad p>3,
\end{equation*}
with the Young's inequality and $R(R-\rho)^{-1}>1$ we deduce that
\begin{equation*}
	I_4 \leq \epsilon G(R) + C(\epsilon) R^{\frac{6-p}{p-3}} (R - \rho)^{-\frac{6-p}{p-3}} \big( R^{(\frac{2s}3 + 4)(\frac p3-1)}\Theta(R)^p + 1 \big)^{\frac 1 {p-3}}, \qquad \epsilon>0.
\end{equation*}

\medskip

\noindent \textit{(ii) $p \geq 4$ case:} 

\medskip

\noindent Using H\"{o}lder's inequality with $\psi = \zeta_{\overline{R},R}$, we have
\begin{equation*}
	I_4 \leq (R-\rho)^{-1} \| B \|_{L^p(B(R))}^{\frac p{2(p-3)}} \|  \psi^3 B \|_{L^3(B(R))}^{\frac {3(p-4)}{2(p-3)}} G(R)^{\frac 12}.
\end{equation*}
Due to $p \geq 4$, we can verify $R^{(\frac {2s}3 + 1)(\frac p3-1)+3}\leq R^{p-2}R^{(\frac{2s}3 + 4)(\frac p3-1)}$ and $R^{-\frac {2p}3 - \frac ps + 3} \leq R^{p-2}$. Applying it to \eqref{Lp_est} yields
\begin{equation*}
	\int_{B(R)} |B|^p \,\mathrm{d}x \leq C R^{p-2} \big( R^{(\frac{2s}3 + 4)(\frac p3-1)}\Theta(R)^p + 1 \big).
\end{equation*}
On the other hand, recalling the definition of $I_2$, we have
\begin{equation*}
	\int_{B(R)} |  \psi^3 B |^3\,\mathrm{d}x = (R-\rho) I_2.
\end{equation*}
Hence, with the Young's inequality we can infer that
\begin{align*}
	I_4 &\leq C R^{\frac {p-2}{2(p-3)}} (R-\rho)^{-\frac {p-2}{2(p-3)}} \big( R^{(\frac{2s}3 + 4)(\frac p3-1)}\Theta(R)^p + 1 \big)^{\frac 1{2(p-3)}} G(R)^{\frac 12} I_2^{\frac {p-4}{2(p-3)}} \\
	&\leq C R(R-\rho)^{-1} \big( R^{(\frac{2s}3 + 4)(\frac p3-1)}\Theta(R)^p + 1 \big)^{\frac 1{p-2}} G(R)^{\frac {p-3}{p-2}} + I_2.
\end{align*}
Noting that $(p-3)/(p-2) < 1$, we use Young's inequality again to have
\begin{equation*}
	I_4 \leq \epsilon G(R) + C(\epsilon) R^{p-2}(R-\rho)^{-(p-2)}\big( R^{(\frac{2s}3 + 4)(\frac p3-1)}\Theta(R)^p + 1 \big) + I_2, \qquad \epsilon>0.
\end{equation*}
Therefore, for any case we have
\begin{equation*}
	I_4 \leq \epsilon G(R) + C(\epsilon) R^{\beta} (R - \rho)^{-\beta} \big( R^{(\frac{2s}3 + 4)(\frac p3-1)}\Theta(R)^p + 1 \big)^{\gamma} + I_2, \qquad \epsilon>0
\end{equation*}
where
\begin{equation*}
	\beta := \max \Big\{ \frac {6-p}{p-3},p-2 \Big\}, \qquad \gamma := \max \Big\{ \frac 1{p-3},1 \Big\}.
\end{equation*}
Taking $\epsilon$ sufficiently small and redefining $\beta$ to satisfy $\beta \geq (s+6)/(s-3)$ also, we deduce
\begin{equation*}
	G(\rho) \leq \frac 12 G(R) + C R^{\beta}(R-\rho)^{-\beta}\big( R^{(\frac{2s}3 + 4)(\frac p3-1)}\Theta(R)^p +1 \big)^{\gamma}.
\end{equation*}
We set $r>1$ and $r \leq \rho < R \leq 2r$. Since we clearly have
\begin{equation*}
	\big( R^{(\frac{2s}3 + 4)(\frac p3-1)}\Theta(R)^p +1\big)^{\gamma} \leq C \bigg( r^{(\frac{2s}3 + 4)(\frac p3-1)} \sup_{\tau \in [r,2r]} \Theta(\tau)^p +1\bigg)^{\gamma},
\end{equation*}
employing the iteration Lemma in \cite[Lemma~3.1]{G83}, we can see
\begin{equation}\label{H1}
	G(\rho) \leq C R^{\beta} (R- \rho)^{-\beta}\bigg( r^{(\frac{2s}3 + 4)(\frac p3-1)} \sup_{\tau \in [r,2r]} \Theta(\tau)^p +1\bigg)^{\gamma}.
\end{equation}
And by letting $\rho = r$ and $R = 2r$, \eqref{ASSS} gives that
\begin{equation}\label{H2}
	G(r) \leq C \bigg( r^{(\frac{2s}3 + 4)(\frac p3-1)} \sup_{\tau \in [r,2r]} \Theta(\tau)^p +1\bigg)^{\gamma} \leq C r^{\gamma(\frac{2s}3 + 4)(\frac p3-1)}
\end{equation}
for sufficiently large $r>1$. To finish showing \eqref{EI}, we estimate $I_4$ in another way. Using H\"{o}lder's inequalities with $\psi = \zeta_{\overline{R},R}$, we can have
\begin{equation*}
	I_4 \leq (R-\rho)^{-1} \| B \|_{L^p(B(R))}^{\frac {3p}{sp-3s+3p}} \| B \|_{L^6(B(R))}^{\frac {sp-3s-3p+18}{sp-3s+3p}} \| \psi^3 B \|_{L^3(B(R))}^{\frac {sp-3s+6p-18}{sp-3s+3p}} G(R)^{\frac 12}.
\end{equation*}
We apply $R^{-\frac {2p}3 - \frac p2 + 3} \leq 1$ to \eqref{Lp_est} and obtain 
\begin{equation*}
	\| B \|_{L^p(B(R))} \leq C R^{(\frac {2s}3 + 1)(\frac 13-\frac 1p)+\frac 3p}\Theta(R) + C.
\end{equation*}
From \eqref{C1} and \eqref{C2} we see that
\begin{gather*}
	\| \psi^3 B \|_{L^3(B(R))} \leq CR^{\frac 13} G(R)^{\frac {3}{s+6}} + CR^{\frac 13+(\frac {11}{12}-\frac1{2s})\frac {6}{s+6}} (R-\rho)^{-\frac {6}{s+6}} G(R)^{\frac 3{2(s+6)}} \\
	+ C R^{\frac 43-\frac 1s} (R - \rho)^{-1}  + C R^{\frac 13+(\frac {11}6 - \frac 1s)\frac {6}{s+6}} (R- \rho)^{-\frac {12}{s+6}}.
\end{gather*}
Then using Young's inequality along with $R(R-\rho)>1$ and $R>1$, we deduce
\begin{equation*}
	\| \psi^3 B \|_{L^3(B(R))} \leq CR^{\frac 13} G(R)^{\frac {3}{s+6}} + C R^{\frac 13+\frac {12}{s+6}} (R- \rho)^{-\frac {12}{s+6}}.
\end{equation*}
Thus, with \eqref{PS} we have
\begin{gather*}
	I_4 \leq C (R-\rho)^{-1} \big( R^{(\frac {2s}3 + 1)(\frac 13-\frac 1p)+\frac 3p}\Theta(R) + R^{-\frac {2p}3 - \frac ps + 3}\big)^{\frac {3p}{sp-3s+3p}} \big( G(R)^{\frac 12} + 1\big)^{\frac {sp-3s-3p+18}{sp-3s+3p}} \\
	\big( R^{\frac 13} G(R)^{\frac {3}{s+6}} + R^{\frac 13+\frac {12}{s+6}} (R- \rho)^{-\frac {12}{s+6}} \big)^{\frac {sp-3s+6p-18}{sp-3s+3p}} G(R)^{\frac 12}.
\end{gather*}
We also set $r \leq \rho < R \leq 2r$ for $r>1$. By $R(R-\rho)>1$ again, we can rewrite it as
\begin{align*}
	I_4 &\leq C R(R-\rho)^{-1} \big( \Theta(R)^{\frac {3p}{sp-3s+3p}} + R^{-\frac {2sp+3p}{3s}} \big) \big(G(R) + R^{\frac {12}{s+6}} (R- \rho)^{-\frac {12}{s+6}} \big) \\
	&\leq C R(R-\rho)^{-1} \bigg( \sup_{\tau \in [r,2r]} \Theta(\tau)^{\frac {3p}{sp-3s+3p}} + r^{-\frac {2sp+3p}{3s}} \bigg) G(2r) + C R^{\frac {s+6}{s-3}} (R- \rho)^{-\frac {s+6}{s-3}}.
\end{align*}
We used in the last inequality that
\begin{equation*}
	\sup_{\tau \in [r,2r]} \Theta(\tau)^{\frac {3p}{sp-3s+3p}} + r^{-\frac {2sp+3p}{3s}} \leq 1
\end{equation*}
which is true for sufficiently large $r > 1$. Recalling \eqref{I123} and taking $\epsilon>0$ sufficiently small, we arrive at
\begin{equation*}
	G(\rho) \leq \frac 12 G(R) + C R^{\frac {s+6}{s-3}}(R-\rho)^{-\frac {s+6}{s-3}} + C R(R-\rho)^{-1} \bigg( \sup_{\tau \in [r,2r]} \Theta(\tau)^{\frac {3p}{sp-3s+3p}} + r^{-\frac {2sp+3p}{3s}} \bigg) G(2r).
\end{equation*}
According to the iteration Lemma in \cite[Lemma~3.1]{G83}, it follows that
\begin{equation*}
	G(\rho) \leq C R^{\frac {s+6}{s-3}}(R-\rho)^{-\frac {s+6}{s-3}} + C R(R-\rho)^{-1} \bigg( \sup_{\tau \in [r,2r]} \Theta(\tau)^{\frac {3p}{sp-3s+3p}} + r^{-\frac {2sp+3p}{3s}} \bigg) G(2r).
\end{equation*}
In particular, for $\rho=r$ and $R=2r$,
\begin{equation}\label{Y}
	G(r) \leq \kappa \bigg( \sup_{\tau \in [r,2r]} \Theta(\tau)^{\frac {3p}{sp-3s+3p}} + r^{-\frac {2sp+3p}{3s}} \bigg) G(2r) + \kappa
\end{equation}
for some $\kappa > 0$. We consider $r>1$ sufficiently large such that
\begin{equation*}
	\kappa \bigg( \sup_{\tau \in [r,2r]} \Theta(\tau)^{\frac {3p}{sp-3s+3p}} + r^{-\frac {2sp+3p}{3s}} \bigg) \leq 2^{-2 \gamma (\frac{2s}3 + 4 )(\frac p3-1)}.
\end{equation*}
Then iterating \eqref{Y} $n$-times yields
\begin{equation*}
	G(r) \leq 2^{-2n \gamma (\frac{2s}3 + 4 )(\frac p3-1)} G(2^n r) + \kappa \sum_{j =0}^{n-1} 2^{-2j \gamma (\frac{2s}3 + 4 )(\frac p3-1)}.
\end{equation*}
Thus, \eqref{H2} implies that
\begin{equation*}
	G(r) \leq 2^{-n \gamma (\frac{2s}3 + 4 )(\frac p3-1)} r^{\gamma (\frac{2s}3 + 4 )(\frac p3-1)} + C.
\end{equation*}
After letting $n \rightarrow \infty$, we pass $r \rightarrow \infty$, which shows \eqref{EI}. 

Following the way we obtained \eqref{L3o1} in the proof of Theorem~\ref{thm1}, we clearly have from \eqref{EI} the same conclusion, and \eqref{L3O1} also follows.

Now, we are ready to finish the proof. Let $r >1$ and $\varphi = \zeta_{r,2r}$. From \eqref{zeta2} we have
\begin{equation*}
	\begin{gathered}
		\int_{B(r)} (|\nabla u|^2 + |\nabla B|^2) \,\mathrm{d}x \leq Cr^{-2} \int_{B(2r)\setminus B(r)} (|u| ^2 + |B|^2) \,\mathrm{d}x \\
		+ Cr^{-1} \!\int_{B(2r)\setminus B(r)} (|u| ^3 + |B|^3) \,\mathrm{d}x + C r^{-1} \bigg( \!\int_{B(2r)} \big| P - P_{B(2r)} \big|^{\frac 32} \,\mathrm{d}x \!\bigg)^{\frac 23} \bigg( \!\int_{B(2r) \setminus B(r)} |u|^3 \,\mathrm{d}x \!\bigg)^{\frac 13} \\
		+ C r^{-1} \int_{B(2r) \setminus B(r)} | B |^2 | \nabla B |\,\mathrm{d}x.
	\end{gathered}
\end{equation*}
To obtain
\begin{equation*}
	\int_{B(r)} (|\nabla u|^2 + |\nabla B|^2) \,\mathrm{d}x \to 0 \qquad \mbox{as} \qquad r \rightarrow \infty,
\end{equation*}
it suffices to show that
\begin{equation*}
	r^{-1} \int_{B(2r) \setminus B(r)} | B |^2 | \nabla B |\,\mathrm{d}x \to 0\ \qquad r \to \infty.
\end{equation*}
Using H\"{o}lder's inequality yields
\begin{equation*}
	r^{-1} \int_{B(2r) \setminus B(r)} | B |^2 | \nabla B |\,\mathrm{d}x\leq C r^{- \frac 12} \bigg( \int_{B(2r)} | B |^6 \,\mathrm{d}x \bigg)^{\frac 13} \bigg( \int_{B(2r) \setminus B(r)} | \nabla B |^2\,\mathrm{d}x \bigg)^{\frac 12}.
\end{equation*}
By \eqref{PS} we have
\begin{equation*}
	r^{-1} \int_{B(2r) \setminus B(r)} | B |^2 | \nabla B |\,\mathrm{d}x\leq C r^{- \frac 12} \bigg\{ \bigg( \int_{B(2r)} | \nabla B |^2 \,\mathrm{d}x \bigg)^{3}+1 \bigg\}^{\frac 13} \bigg( \int_{B(2r) \setminus B(r)} | \nabla B |^2\,\mathrm{d}x \bigg)^{\frac 12}.
\end{equation*}
From \eqref{EI} we obtain what we desired. This completes the proof.\\
\ \\
\hspace{0.5cm}
$$\mbox{\bf Acknowledgements}$$
Chae's research  was partially supported by NRF grants 2021R1A2C1003234, and by the Chung-Ang University research grant in 2019.
Wolf has been { supported by }NRF grants 2017R1E1A1A01074536.
The authors declare that they have no conflict of interest.

\end{document}